\documentclass{amsart}
\usepackage{graphicx,,amssymb}
\usepackage{amsfonts}
\usepackage{mathrsfs}
\usepackage[all]{xy}
\usepackage{overpic}
\usepackage{bbm}
\newtheorem{thm}{Theorem}[section]
\newtheorem{lemma}[thm]{Lemma}
\newtheorem{prop}[thm]{Proposition}
\newtheorem{cor}[thm]{Corollary}

\theoremstyle{definition}
\newtheorem{defi}[thm]{Definition}
\newtheorem{example}[thm]{Example}

\theoremstyle{remark}
\newtheorem{remark}[thm]{Remark}
\numberwithin{equation}{section}

\newcommand{\hh}{{\bf{h}}}

\newcommand{\PP}{\mathbb{P}}      
\newcommand{\OO}{\mathcal{O}}     
\newcommand{\F}{\mathcal{F}}      
\newcommand{\G}{\mathcal{G}}      
\newcommand{\T}{\mathcal{T}}      
\newcommand{\D}{\mathcal{D}}      
\newcommand{\C}{\mathcal{C}}      
\newcommand{\Gammahat}{{\widehat{\Gamma}}}

\newcommand{\Om}{\Omega}

\newcommand{\CC}{\mathbb{C}}       
\newcommand{\I}{\mathbbm{1}}
\newcommand{\Z}{\mathbb{Z}}       
\newcommand{\g}{\mathfrak{g}}

\newcommand{\s}{\mathfrak{sl}}
\newcommand{\K}{\mathcal{K}}
\newcommand{\MM}{\Mod_{G}}

\DeclareMathOperator{\Ind}{Ind}
\DeclareMathOperator{\Hom}{Hom}
\DeclareMathOperator{\RHom}{RHom}
\DeclareMathOperator{\Ext}{Ext}
\DeclareMathOperator{\End}{End}
\DeclareMathOperator{\id}{id}
\DeclareMathOperator{\Path}{Path}
\DeclareMathOperator{\Rep}{Rep}

\DeclareMathOperator{\Coh}{Coh}
\DeclareMathOperator{\QCoh}{QCoh}
\DeclareMathOperator{\Proj}{Proj}
\DeclareMathOperator{\Mod}{Mod}
\begin{document}
\title{Quantum McKay Correspondence and Equivariant Sheaves on $\PP_{q}^{1}$}
\author{A. Kirillov Jr.}
 \address{Department of Mathematics, SUNY at Stony Brook, 
            Stony Brook, NY 11794, USA}
    \email{kirillov@math.sunysb.edu}
    \urladdr{http://www.math.sunysb.edu/\textasciitilde kirillov/} 
\author{J. Thind}
 \address{Department of Mathematics, University of Toronto, 
            Toronto, ON, Canada}
    \email{jthind@math.utoronto.ca}
    \urladdr{http://www.math.utoronto.ca/jthind} 
\maketitle
\begin{abstract}
In this paper, using the quantum McKay correspondence, we construct the ``derived category" of $G$-equivariant sheaves on $\PP^{1}_{q}$ - the quantum projective line at a root of unity. More precisely, we use the representation theory of $U_{q} (\s_{2})$ at root of unity to construct an analogue of the symmetric algebra and the structure sheaf. The analogue of the structure sheaf is, in fact, a complex, and moreover it is a dg-algebra. Our derived category arises via a triangulated category of $G$-equivariant dg-modules for this dg-algebra. We then relate this to representations of the quiver $(\Gamma , \Om)$, where $\Gamma$ is the $A,D,E$ graph associated to $G$ via the quantum McKay correspondence, and $\Om$ is an orientation of $\Gamma$. As a corollary, our category categorifies the corresponding root lattice, and the indecomposable sheaves give the corresponding root system.
\end{abstract}

\section{Introduction}

The McKay correspondence gives a bijection between the finite subgroups of $\text{SU}(2)$, and affine $A,D,E$ Dynkin graphs. There are several geometric statements related to this correspondence. In particular, one can look at $\CC^{2} /G$ and its minimal resolution $\pi: X \to \CC^{2} /G$. In this case, the exceptional fibre $\pi^{-1} (0)$ is a collection of $\PP^{1}$'s whose intersection graph is the corresponding finite $A,D,E$ graph. At the level of categories, Kapranov and Vasserot (\cite{kapvas}) showed there is an equivalence of derived categories $D^{b} (X) \to \D^{b} (\CC^{2} / G)$, and they also showed that this categorified the corresponding root lattice. (I.e. if $G$ corresponds to the Dynkin diagram $\Gamma$, one can realize the associated root lattice as the Grothendieck group of $\D^{b} (\CC^{2} / G)$.)  Bridgeland, King and Reid generalized this to higher dimension, stating the McKay correspondence as a derived equivalence and proving the 3 dimensional case in \cite{bkr}.

Another approach, due to Kirillov Jr. (\cite{kirillov}), was to consider $\overline{G}$-equivariant sheaves on $\PP^{1}$, where $\overline{G}= G / \pm \mathbb{I}$. The derived category of $\overline{G}$-equivariant sheaves on $\PP^{1}$, denoted $\D_{\overline{G}} (\PP^{1})$, again categorifies the corresponding root lattice, and one can realize the roots as the indecomposable sheaves. (In fact, one realizes the real roots as the locally free sheaves, and the imaginary roots as the torsion sheaves.)

There is a quantum version of the McKay correspondence, which gives a bijection between ``finite subgroups" of $U_{q} (\s_{2})$ and finite $A,D,E$ Dynkin graphs. (See Section~\ref{s:qmckay} or \cite{kirillov-ostrik}  for more details.) In this paper we establish a ``quantum" version of the construction in \cite{kirillov}. Namely, we define the derived category of $G$-equivariant sheaves on $\PP^{1}_{q}$ - the quantum projective line, where $q$ is a root of unity.

A natural question is ``What is meant by $\PP^{1}_{q}$?". It should be some sort of non-commutative space, analogous in some way to $\PP^{1}$. There are several approaches to defining non-commutative spaces, but the one we take is categorical. In particular, we define the quantum space by defining its derived category of sheaves. The results of Gabriel (\cite{gabriel}) and Rosenberg (\cite{rosen}) roughly state that a scheme $X$ is the same thing as its category of quasi-coherent sheaves $\QCoh (X)$. Rosenberg and Lunts (\cite{rosen}, \cite{rosen-lunts}) have used this approach to define quantum flag varieties via the representation theory of the corresponding quantum group. The starting point for this perspective is the result of Serre (\cite{serre}), stating that the category of sheaves on $\PP^{n}$ is equivalent to the category of graded $S$-modules, modulo an equivalence $\sim$. The category of sheaves on the flag variety $G/B$ can be realized as $S-\Mod / \sim$,  for $S = \oplus V_{\lambda}$, where the $V_{\lambda}$ are the irreducible representations of $G$. The basic idea in the quantum version is to take the graded algebra $S_{q} = \oplus V_{\lambda}$, where the $V_{\lambda}$ are the irreducible representations of $U_{q} (\g)$ for $q$ generic, and define the category $\Proj (S_{q}) = S_{q}-\Mod/ \sim$. They then define the quantum $G/B$ as the category $\Proj (S_{q})$ thought of as ``$\QCoh (G/B)$". In this paper, we take a similar approach, defining an algebra $S_{q}$, and dg-algebra $\OO_{q}$, from the irreducible representations of $U_{q} (\s_{2})$ at root of unity. However, instead of realizing the derived category of sheaves as a quotient of the derived category of $S_{q}$-modules, we realize it instead as a suitable subcategory of $\D (\OO_{q} -\Mod_{G})$ - the derived category of $G$-equivariant dg $\OO_{q}$-modules.

Backelin and Kremnitzer \cite{back-krem} have also defined a version of the category of sheaves on quantum flag varieties, however we do not take their approach.

A few surprising things happen when working with $q$ being a root of unity. First, in order to define what is naturally the analogue of the structure sheaf $\OO$ we are naturally forced to consider complexes. In some sense this forces us to first define $\D_{G} (\PP^{1}_{q})$, the derived category of $G$-equivariant coherent sheaves on $\PP^{1}_{q}$, rather than defining an abelian category, then taking its derived category.

The layout of the paper is as follows. We first recall the classical McKay correspondence, some basic facts about the derived category of sheaves on $\PP^{1}$, as well as the geometric construction of Kirillov Jr in Section~\ref{s:mckay}. In Section~\ref{s:sheaves} we give an alternate description of the derived category of $G$-equivariant coherent sheaves on $\PP^{1}$, as a full subcategory inside $\D (S-\Mod_{G})$. In Section~\ref{s:qmckay} we recall the quantum McKay correspondence. We then define the analogue of the symmetric algebra $S_{q}$, and the structure sheaf $\OO_{q}$ in Section~\ref{s:qproj}. In Section~\ref{s:qsheaves} we define $\D_{G} (\PP^{1}_{q})$, the derived category of $G$-equivariant coherent sheaves on $\PP^{1}_{q}$, where $G$ is a finite subgroup of $U_{q} (\s_{2})$. With this category in place, we define some natural objects, which turn out to give a complete list of indecomposables. We then relate this to the construction given in \cite{kirillovthind2}. This allows us to view the category constructed in \cite{kirillovthind2} as a combinatorial model for the category of equivariant sheaves on $\PP^{1}_{q}$. In the final section, we construct equivalences to derived categories of quiver representations, or more precisely, to their 2-periodic quotient. Using these equivalences, we show that $\D_{G} (\PP^{1}_{q})$ categorifies the root lattice (and root system) of the corresponding Dynkin graph.


\section{Classical McKay Correspondence}\label{s:mckay}

The classical McKay correspondence gives a bijection between the finite subgroups $G \subset \text{SU}(2)$ and affine $A,D,E$ Dynkin diagrams. The bijection goes as follows. Given $G \subset \text{SU}(2)$, we can construct a graph $\Gamma(G)$ whose vertex set is the set of isomorphism classes of irreducible representations of $G$. Let $n_{ij} = \dim \Hom (X_{i} , \CC^{2} \otimes X_{j} )$. Join $[X_{i}]$ and $[X_{j}]$ by $n_{ij}$ edges. (Note, $n_{ij} = n_{ji}$.)

McKay showed that $G \mapsto \Gamma(G)$ provides a bijection between finite subgroups of $\text{SU}(2)$ and affine $A,D,E$ Dynkin graphs (see \cite{mckay}).

This correspondence can also be approached geometrically, in terms of the associated singularity $\CC^{2}/G$ and its minimal resolution. This has been developed by several authors, notably Gonzales-Sprinberg and Verdier \cite{g-sv} (in terms of K-Theory), Kapranov and Vasserot \cite{kapvas} (in terms of derived categories), and Bridgeland, King and Reid \cite{bkr} (generalization to higher dimension).

There is another geometric approach, due to Kirillov Jr., via the derived category of equivariant sheaves on $\PP^{1}$.  This is the starting point for this paper, so we will briefly review that construction. For full details see \cite{kirillov}.

\subsection{Equivariant sheaves on $\PP^{1}$.}

Let $G \subset \text{SU}(2)$ be a finite subgroup with irreducible representations $X_{i}$, and let $V$ denote the standard 2-dimensional representation of $\text{SU}(2)$. Set $\overline{G} = G/{\pm Id}$ (here, we exclude the cyclic group $G=C_{2n-1}$, corresponding to $\hat{A}_{2n-2}$, though that case can be dealt with by hand). Let $\C = \Coh_{\overline{G}} (\PP^{1})$ be the category of $\overline{G}$-equivariant coherent sheaves on $\PP^{1}$, and let $\D = \D^{b} (\C)$ be its derived category. Before describing the results of \cite{kirillov} we need some preliminaries.

Given a bipartite graph $\Gamma$, let $p: \Gamma \to \Z_{2}$ be a parity function coming from a bipartite splitting. 
Define a quiver structure on $\Gamma \times \Z$ by joining $(i,n)$ to $(j,n+1)$ if $i,j$ are connected in $\Gamma$. Then $\Gamma \times \Z = \Gammahat^{0} \sqcup \Gammahat^{1}$ as disjoint quivers, where
$\Gammahat^{p} = \{ (i,n) \subset \Gamma \times \Z \ | \ p(i) + n = p \mod 2 \}$. 

In the case that $\Gamma$ is finite $A,D,E$ with Coxeter number $h$, define cyclic versions $\Gamma \times \Z_{2h}$, $\Gammahat_{cyc}^{p} = \{ (i,n) \subset \Gamma \times \Z_{2h} \ | \ p(i) + n = p \mod 2 \}$, with arrows $(i,n) \to (j,n+1)$ for $i$ connected to $j$ in $\Gamma$. We will typically drop the superscript and denote $\Gammahat^{0}$ and $\Gammahat^{0}_{cyc}$ by $\Gammahat$ and $\Gammahat_{cyc}$ respectively.

\begin{figure}[ht]
    \centering
        \includegraphics[height=2.50in]{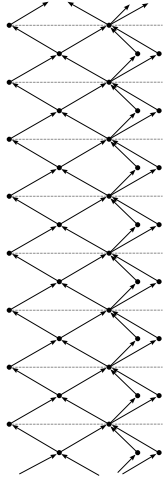}
        \caption{The quiver $\Gammahat$ for graph $\Gamma = D_{5}$. For $D_{5}$, the Coxeter number is 8, so we obtain $\Gammahat_{cyc}$ from this figure by identifying the outgoing arrows at the top level, with the incoming arrows at the bottom level.}\label{f:Ihat-D5}
        \end{figure}

\begin{example}
    For the graph $ \Gamma = D_{5}$ the quiver $\Gammahat$ is shown in
    Figure~\ref{f:Ihat-D5}. 
\end{example}

Define a translation $\tau_{\Gammahat} : \Gammahat \to \Gammahat$ by $\tau_{\Gammahat} (i,n) = (i,n+2)$.

A function $\hh: \Gamma \to \Z$ is called a height function if $\hh(i) + p(i) = 0 \mod 2$ and $\hh(i) = \hh(j) \pm 1$ for $i$ connected to $j$. A height function gives an orientation $\Om_{\hh}$ of $\Gamma$, defined by $i \to j$ in $\Om_{\hh}$ if $i$ is connected to $j$ and $\hh(j) = \hh(i) +1$. Moreover, a height function gives an embedding of $\Gamma$ into $\Gammahat$ as a full connected subquiver, by sending $i \mapsto (i, \hh(i))$. The image is denoted by $\Gamma_{\hh}$ and called a ``slice". Note that as quivers, $\Gamma_{\hh} \simeq (\Gamma , \Om_{\hh} )$.

If $i \in \Gamma$ is a source (or sink) in $\Om_{h}$ we can define a new height function $s_{i}^{\pm} \hh$ by 
$$s_{i}^{\pm} \hh (j) = \begin{cases}
\hh (j) &\text{ if } j\neq i \\
\hh(i) \pm 2 &\text{ if } j=i. \\
\end{cases}$$

In this case the orientation determined by $s_{i}^{\pm} \hh$ is obtained from $\Om_{\hh}$ by reversing all the arrows at $i$.

For each edge $e:(i,n) \to (j,n+1)$ in $\Gammahat$ choose $\epsilon (e) \in \Z_{2}$ so that $\epsilon (e) + \epsilon ( \overline{e} ) = 0$, where $\overline{e} : (j,n+1) \to (i,n+2)$. For each vertex $v \in \Gammahat$ define $\theta_{v}$ by
\begin{equation}\label{e:mesh}
\theta_{v} =\sum_{e: s(e)= v}  \epsilon (e)  \overline{e} e
\end{equation}

where $s(e)$ denotes the source of the edge $e$. Let $J$ be the ``mesh" ideal of the path algebra of $\Gammahat$ generated by the $\theta_{v}$'s. 

With this in place, we can now summarize the results of \cite{kirillov}.

\begin{thm}\label{t:Gsheaves} Let $G$ be a finite subgroup of $\text{SU}(2)$, and $\Gamma = \Gamma(G)$ be the associated graph. Let $\C$ be the category of $\overline{G}$-equivariant coherent sheaves on $\PP^{1}$ and let $\D$ be its bounded derived category.
\begin{enumerate}
\item Let $\K$ denote the Grothendieck group of $\C$. \\
Set $< \F , \G > = \dim \Hom (\F, \G) - \dim \Ext^{1} (\F , \G)$ and $( \F, \G ) = < \F , \G > + < \G, \F >$.\\ 
Then $\K$ can be identified with the root lattice of $\Gamma$, the form $( \cdot \ , \cdot )$ above can be identified with its bilinear form, and the set $\Ind = \{ [X] \in \K \ | \ X \text{ is indecomposable in } \D \}$ can be identified with the set of roots.
\item The map on $\K$ induced by the ``twist" $\F \mapsto \F (-2)$ gives a canonical Coxeter element for this root system.
\item The classes of indecomposable locally free sheaves are of the form $X_{i} (n) := X_{i} \otimes \OO(n)$ for $p(i)+n \equiv 0 \mod 2$, where $X_{i}$ is an irreducible representation of $G$. They are in bijection with vertices of the quiver $\Gammahat$, so that the arrows correspond to indecomposable morphisms $X_{i} (n) \to X_{j} (m)$. 
\item For any height function $\hh: \Gamma \to \Z$, there is an equivalence of triangulated categories 
$$\Phi_{\hh} : \D \to \D^{b} (\Gamma, \Om_{\hh}^{op})$$
given by a tilting object in $\D$.
\item If $i \in \Gamma$ is a source (or sink) for $\Om_{\hh}^{op}$, then the following diagram commutes:
$$\xymatrix{
&  \D^{b} (\Gamma ,\Om_{\hh}^{opp})   \ar[dd]^{S^{\pm}_{i}} \\
\D \ar[dr]_{\Phi_{s^{\pm}_{i} \hh}} \ar[ur]^{\Phi_{\hh}} & \\
& \D^{b} (\Gamma ,\Om^{opp}_{s^{\pm}_{i} \hh}) \\
}$$
where $S^{\pm}_{i}$ is the derived BGP reflection functor.
\item The short exact sequence
\begin{equation*}
0\to \OO \to \OO \otimes V^{*} \to \OO (2) \otimes \Lambda^{2} V^{*} \to 0
\end{equation*}
gives a short exact sequence of indecomposable locally free sheaves
\begin{equation*}
0 \to X_{i} (n) \to \bigoplus_{i-j} X_{j} (n+1) \to X_{i} (n+2) \to 0
\end{equation*}
which induces the full set of relations in $\K$.
\item $\Hom (X_{i} (n) , X_{j} (m)) = \Path_{\Gammahat} ((i,n), (j,m)) / J$ where $J$ is the mesh ideal defined above.
\item Serre Duality: $\big{(} \Hom (\F , \G) \big{)}^{*} \simeq \Ext^{1} (\G , \F (-2))$.
\end{enumerate}
\end{thm}

\section{Alternate description of sheaves on $\PP^{1}$}\label{s:sheaves}

Before giving an alternate description of the derived category of equivariant sheaves on $\PP^{1}$, let us first recall the standard one. Let $\PP^{1} = \PP (V^{*})$, where $V$ is the canonical 2 dimensional representation of $\text{SU}(2)$. Let $S= S(V)$ denote the symmetric algebra of $V$. Note that $S = \bigoplus V_{n}$, where $V_{n}$ is the irreducible $\text{SU}(2)$ representation of highest weight $n$. Let $S-\Mod^{f}_{gr}$ denote the category of finitely generated, graded $S$-modules. For a graded $S$-module $M$, denote by $M_{k}$ its k-th graded component. There is a twist functor $(n)$ on $S-\Mod^{f}_{gr}$ defined by $M(n)_{k} = M_{n+k}$. Define an equivalence $\sim$ on $S-\Mod^{f}_{gr} $ by $M\sim N$ if $M_{n} = N_{n}$ for all $n\gg 0$.

Consider the category $\Coh (\PP^{1})$ of coherent sheaves. Recall the exact functor $(n) : \Coh \to \Coh$ defined by $\F \mapsto \F (n) = \F \otimes \OO(n)$ (the ``twist"). There is a left exact functor $\Gamma_{*} : \Coh(\PP^{1}) \to S-\Mod^{f}_{gr}$ defined by $\Gamma_{*} (\F) = \bigoplus \Gamma (\F(n))$, where $\Gamma$ is the global sections functor. There is a functor $\widetilde{} : S-\Mod^{f}_{gr} \to \Coh(\PP^{1})$ taking a graded module to a coherent sheaf. We will not define this functor, but a definition can be found in \cite{hartshorne} (p.116). We have that $\widetilde{\Gamma_{*} (\F)} \simeq \F$, and $\Gamma_{*} (\widetilde{M})_{k} = M_{k}$ in high enough homogeneous degree $k$.

The following result (going back to Serre \cite{serre}) describes the category of coherent sheaves on $\PP^{1}$ in terms of graded modules over $S$. (For more details, see \cite{hartshorne} or \cite{serre}.)

\begin{thm}
There is an equivalence of categories between $\Coh (\PP^{1})$ and $S-\Mod_{gr}^{f} / \sim $, given by $\Gamma_{*}$.
\end{thm}

Moreover, there is an equivariant version of this statement. If $G$ is a finite subgroup of $\text{SU}(2)$, then we can realize $\Coh_{G} (\PP^{1})$ as a suitable quotient of the category of finitely generated, graded, $G$-equivariant $S$ modules, denoted by $S-\Mod_{G}$. Objects  can be written as $M = \bigoplus M_{k}$, where each $M_{k}$ is a $G$-module, and the map $S\otimes M \to M$ is a graded map of $G$-modules.

These descriptions also allow us to describe the corresponding derived categories $\D (\PP^{1})$ and $\D_{G} (\PP^{1})$ as certain quotients of $\D (S-\Mod_{gr})$ and $\D (S-\Mod_{G})$, respectively. However, for the purposes of this paper, we would like to realize them as certain subcategories instead.

\subsection{Short Exact Sequences and Triangles}

On $\PP^{1}$ there is a natural short exact sequence of sheaves 
\begin{equation}\label{e:oseq}
0\to \OO \to \OO(1) \otimes V^{*} \to \OO (2) \otimes \Lambda^{2} V^{*} \to 0.
\end{equation}
In particular, there is an induced triangle in $\D (\PP^1)$. Since $\OO$ has a natural $G$-equivariant structure, we also get a triangle in $\D_{G} (\PP^1)$.
Let $X_{i}$ be an irreducible representation of $G$. Then tensoring with the locally free sheaf $\OO(n) \otimes X_{i}$, and choosing an identification of $\Lambda^{2} V^{*} \simeq \CC$, gives a short exact sequence of $G$-equivariant locally free sheaves
$$0 \to \OO(n) \otimes X_{i} \to \OO(n+1) \otimes V^{*} \otimes X_{i} \to \OO(n+2) \otimes X_{i} \to 0,$$
and hence a triangle in $\D_{G} (\PP^1)$.
Since $X_{i} \otimes V = \bigoplus_{i-j} X_{j}$ we can rewrite this sequence as 
\begin{equation}\label{e:fundseq}
0 \to X_{i} (n) \to \bigoplus_{i-j} X_{j} (n+1) \to X_{i} (n+2) \to 0
\end{equation}
where the notation ``$i-j$" means all vertices $j$ connected to $i$ in $\Gamma$.

For any complex of sheaves $\F^{\bullet}$ there is a triangle $\F^{\bullet} \to \F^{\bullet} (1) \otimes V^{*} \to \F^{\bullet} (2) \to \F^{\bullet} [1]$, obtained by tensoring $\F^{\bullet}$ with the triangle coming from Equation~\ref{e:oseq}.

Applying the derived functor $R \Gamma_{*}$ to the triangle in Equation~\ref{e:oseq} gives us a triangle in $\D (S-\Mod_{G})$:
\begin{equation}\label{e:sseq}
S \to S (1) \otimes V^{*} \to S(2) \to S[1].
\end{equation}

In homogeneous degree $k \geq -1$ this is the triangle coming from the short exact sequence $0 \to V_{k} \to V_{k+1} \otimes V^{*} \to V_{k+2} \to 0$.

\begin{defi}\label{d:DG}
Let $\D_{G} \subset \D(S-\Mod_{G})$ be the full subcategory of objects $M^{\bullet}$ such that
$$M^{\bullet} \to M^{\bullet}(1) \otimes V \to M^{\bullet}(2) \to M^{\bullet} [1]$$
is a triangle. (I.e. the result of tensoring the triangle coming from Equation~\ref{e:sseq} with $M^{\bullet}$ is again a triangle.)
\end{defi}

\begin{thm}\label{t:altsheaves} \
\begin{enumerate}
\item $\D_{G}$ is a triangulated subcategory of $\D(S-\Mod_{G})$.
\item There is an equivalence of categories $\D_{G} (\PP^{1}) \to \D_{G}$ given by $R\Gamma_{*}$.
\end{enumerate}
\end{thm}

This result should be known to experts, but we have not been able to find it in the literature, so we shall give the proof. We shall need some preliminary results.

\begin{lemma}\label{l:prelim} \
\begin{enumerate}
\item For any complex of sheaves $\F^{\bullet} \in \D_{G} (\PP^{1})$, $R\Gamma_{*} (\F^{\bullet}) \in \D_{G}$. (I.e. the image of the functor $R\Gamma_{*}$ lies in $\D_{G}$.)
\item If $M^{\bullet}, N^{\bullet} \in \D_{G}$, and for some $n_{0} > 0$, we have that $M^{\bullet}_{n} \sim N^{\bullet}_{n}$ for all $n > n_{0}$, then, in fact, $M^{\bullet}_{n} \sim N^{\bullet}_{n}$ for all $n$.
\end{enumerate}
\end{lemma}

\begin{proof} \
\begin{enumerate}
\item For any complex of sheaves $\F^{\bullet}$ there is a triangle $\F^{\bullet} \to \F^{\bullet} (1) \otimes V^{*} \to \F^{\bullet} (2) \to \F^{\bullet} [1]$,  obtained by tensoring $\F^{\bullet}$ with the triangle coming from Equation~\ref{e:oseq}. Since $R\Gamma_{*}$ is an exact functor, it takes triangles to triangles. Hence $R\Gamma_{*} (\F^{\bullet}) \in \D_{G}$.
\item Since $M^{\bullet} , N^{\bullet} \in \D_{G}$, and $M^{\bullet}_{n} \sim N^{\bullet}_{n}$ for all $n > n_{0}$, then provided that $n \geq n_{0}$ there are triangles and quasi-isomorphisms $g,h$ as in the diagram below.

$$
\xymatrix{
M^{\bullet}_{n}  \ar[r] \ar[d]_{\exists f} & M^{\bullet}_{n+1} \otimes V \ar[r] \ar[d]_{g} & M^{\bullet}_{n +2} \ar[r] \ar[d]_{h} & M^{\bullet}_{n} [1] \ar[d]_{f[1]} \\
N^{\bullet}_{n}  \ar[r] & N^{\bullet}_{n+1} \otimes V \ar[r] & N^{\bullet}_{n +2} \ar[r] & N^{\bullet}_{n} [1] \\
}$$
This diagram can be completed with a map $f: M^{\bullet}_{n} \to N^{\bullet}_{n}$, which is also a quasi-isomorphism. (See \cite{gelfman} Corollary 4, p. 242.) Taking $n=n_{0}$, we see that $M^{\bullet}_{n_{0}} \simeq N^{\bullet}_{n_{0}}$. Continuing in this way, we can show that for any $n \leq n_{0}$,  $M^{\bullet}_{n} \simeq N^{\bullet}_{n}$. Hence $M^{\bullet}_{n} \sim N^{\bullet}_{n}$ for all $n$.

\end{enumerate}
\end{proof}

We can now complete the proof of Theorem~\ref{t:altsheaves}
\begin{proof}
For (1), since $\D_{G}$ is full, once we verify that any morphism $f:M^{\bullet} \to N^{\bullet}$ can be completed to a triangle, the rest of the axioms of a triangulated category follow easily. To do this, we just verify that if we complete the triangle in $\D(S-\Mod_{G})$, the completion actually lies in $\D_{G}$. Checking this is straightforward and follows from the definition of cone of a morphism.
For (2), first note that $R\Gamma_{*} (\widetilde{M^{\bullet}})$ is quasi-isomorphic to $M^{\bullet}$ in high enough homogeneous degree. Since $R\Gamma_{*} (\widetilde{M^{\bullet}}) \in \D_{G}$ (by Part (1) of Lemma~\ref{l:prelim}), Part (2) of Lemma~\ref{l:prelim} implies that $R\Gamma_{*} (\widetilde{M^{\bullet}})$ is quasi-isomorphic to $M^{\bullet}$ in all degrees, and hence $R\Gamma_{*} (\widetilde{M^{\bullet}}) \simeq M$.  
\end{proof}

\section{Quantum McKay Correspondence}\label{s:qmckay}
In this section we review the quantum analog of McKay correspondence. Such a correspondence should be between ``finite subgroups of $U_{q} (\s_{2})$" and finite A,D,E Dynkin diagrams. This first appeared in Ocneanu's work and is framed in the language of subfactors (\cite{oc}). The version we present is in the language of algebras in tensor categories, and comes from  the representation theory of $U_{q} (sl_2)$ at root of unity. It is due to Kirillov Jr. and Ostrik \cite{kirillov-ostrik}, and the construction is much closer in spirit to the original McKay correspondence. (In particular, the vertices of the graph $\Gamma(G)$ are the isomorphism classes of the simple objects in a category analogous to $\Rep (G)$, and the edges are defined via analogous $\Hom$ spaces.) 

Let $q=e^{\frac{\pi i}{h}}$ be a root of unity, and let $\C$ denote the semisimple quotient of the category of representations of $U_{q} (\s_{2})$. $\C$ is a fusion category with finitely many simple objects $V_{0} = \I , V_{1}, \ldots, V_{h-2}$.

Recall the fusion rule for $\C$. It is given by 
\begin{equation}\label{e:fusionrule}
V_{n} \otimes V_{m} = \bigoplus_{k = | n-m |}^{N_{n,m}} V_{k} 
\end{equation}
where $N_{n,m}= \text{min} \{ n+m , 2(h-2)-(n+m) \}$ and $k+n+m \in 2\Z$.

\begin{defi} 
An algebra in $\C$ is an object $A \in \C$ with maps $\I \to A$, $A \otimes A \to A$ satisfying the obvious axioms (which are the natural generalizations of the axioms of an algebra with unit).
\end{defi}

\begin{defi}
An $A$-module in $\C$ is a pair $(M, \mu_{M})$, where $M \in \C$ and $\mu_{M} : A \otimes M \to M$, which satisfies the obvious axioms (which are the natural generalizations of the axioms of a module over an algebra). 
\end{defi}

Define $\Hom (M_{1} , M_{2})$ in the obvious way and denote by $\Rep (A)$ the category of modules over $A$.

There is a tensor product $$M_{1} \otimes_{A} M_{2} = M_{1} \otimes_{\C} M_{2} / \text{Im}(\mu_{1} -\mu_{2}),$$ where $\mu_{i} : A \otimes_{\C} M_{1} \otimes_{\C} M_{2} \to M_{1} \otimes_{\C} M_{2}$ are given by $\mu_{1} = \mu_{M_{1}} \otimes id$ and $\mu_{2} = (\id_{M_{1}} \otimes \mu_{M_{2}}) R_{A M_{1}}$  and $R$ is the braiding in $\C$.

There are tensor functors $F: \C \to \Rep(A)$ given by $F(X) = A\otimes_{\C} X$ and $\mu_{F(X)} = \mu \otimes id$, and $G: \Rep (A) \to \C$, the forgetful functor. These functors are exact and adjoint to each other.

The functor $F$ makes $\Rep(A)$ a module category over $\C$.

Given an algebra $A$, construct a graph $\Gamma (A)$ whose vertices are the isomorphism classes of simple objects $[X_{i}]$ in $\Rep (A)$, with adjacency matrix given by $n_{ij} = \dim \Hom (X_{i} , F(V_{1}) \otimes_{A} X_{j})$.

\begin{thm}\cite{kirillov-ostrik}\label{t:qmckay}
The map $A \mapsto \Gamma (A)$ gives a bijection between a certain class of $\C$-algebras (rigid, commutative algebras with $\theta_{A} = id$, see \cite{kirillov-ostrik} for details) and finite Dynkin diagrams of type $A_{n}, D_{2n}, E_{6}, E_{8}$.\\
\end{thm}

\begin{example}
Consider the case $q=e^{\frac{\pi i}{h}}$, and the trivial algebra $A= \I$. Then $\Rep(A) = \C$, so the simple objects are just $V_{0}, V_{1}, \ldots , V_{h-2}$.  A simple calculation of $V_{k} \otimes V_{1}$ in $\C$ shows that the corresponding graph is of type $A_{h-1}$.
\end{example}

Note that in this construction $h$ varies as the Coxeter number of the corresponding root system, and there may be more than one algebra for the same root of unity. For example, if $q=e^{\frac{\pi i }{10}}$ there are two algebras in $\C$; one corresponding to $A_{9}$ and one corresponding to $D_{6}$.

For the rest of this paper we will abuse language, and notation, and call such algebras in $\C$ subgroups and denote them by $G$. We will denote by $\Mod_{G}$ the corresponding module category of representations of G.

\begin{defi}
Let $S$ be an algebra in $\C$.  A $G$-equivariant $S$-module is an object $M\in \Mod_{G}$, such that $F$ gives $M$ the structure of an $S$-module. More precisely, there are maps of $G$-modules $F(S) \otimes M \to M$, satisfying the axioms of a module. 
\end{defi}

\begin{defi}
A dg-algebra $S$ in $\C$ is a complex of objects $S = \oplus S^{i}$ ($S^{i} \in \C$), together with a graded map $S \otimes S \to S$ (i.e. $S^{i} \otimes S^{j} \to S^{i+j}$) and inclusion $\I \to S^{0}$, satisfying the analogues of the axioms of a dg-algebra. (I.e one should replace vector spaces with objects in $\C$, and linear maps with morphisms, then ask for commutative diagrams that are the analogues of the commutative diagrams that encode the structure of a differential graded algebra.) In particular, $S$ should be a graded, associative algebra in $\C$, and the differential on $S$ should act as a derivation with respect to multiplication: $d_{S} (xy) = (d_{S} x)y + (-1)^{|x|} x (d_{S} y)$.
\end{defi}
\begin{defi}
A dg-module $M$ for a dg-algebra $S$, is a complex of objects $M= \oplus M^{i}$, together with maps $S^{j} \otimes M^{i} \to M^{i+j}$, satisfying the obvious analogue of the axioms of a dg-module over a dg-algebra. (Again, replace vector spaces with objects in $\C$, and linear maps with morphisms, then require the appropriate diagrams to commute.) In particular, $M$ is a graded module for $S$, and there is the usual compatibility between differentials: $d_{M} (s.m) = (d_{S} s).m + (-1)^{|s|} s. d_{M} m$.
\end{defi}

\begin{defi}
Let $S$ be a dg-algebra in $\C$. A $G$-equivariant dg $S$-module is a complex of objects $M = \oplus M^{i}$ in $\Mod_{G}$, such that $F$ gives $M$ the structure of a dg $S$-module. More precisely, there are maps $F(S^{i}) \otimes_{G} M^{j} \to M^{i+j}$, satisfying axioms analogous to those of a dg-module. (Recall that $F$ is a tensor functor, so that $F(S^{i}) \otimes_{G} F(S^{j}) = F(S^{i} \otimes_{\C} S^{j})$.) Denote the category of $G$-equivariant dg $S$-modules by $S-\Mod_{G}$.
\end{defi}

We will also consider graded versions of these, where in addition to the homological grading $S= \bigoplus S^{i}$, $M= \bigoplus M^{i}$, there is an additional homogeneous grading: $S^{i} = \bigoplus \bigoplus S^{i}_{j}$, $M^{i}= \bigoplus \bigoplus M^{i}_{j}$. In this case, all the maps must also respect the homogeneous grading: $S^{i}_{k} \otimes S^{j}_{l} \to S^{i+j}_{k+l}$, $S^{i}_{k} \otimes M^{j}_{l} \to M^{i+j}_{k+l}$.

Though we have not specified any particular homological, or homogeneous gradings in our definitions, for the rest of the paper we will consider the case of $\Z_{2}$-homological grading, and $\Z_{2h}$-homogenous grading: 
$$S=\bigoplus_{i \in \Z_{2}} \bigoplus_{ j \in \Z_{2h}} S^{i}_{j}$$ 
$$M=\bigoplus_{i \in \Z_{2}} \bigoplus_{ j \in \Z_{2h}} M^{i}_{j}.$$

\section{Quantum Projective Line}\label{s:qproj}

There are several points of view about how to define a non-commutative space. We shall take the categorical approach described in the introduction. Namely, we introduce an analogue $S_{q}$ of the symmetric algebra, and define an analogue $\OO_{q}$ of the structure sheaf. It turns out that the natural candidate for $\OO_{q}$ is a complex of $S_{q}$-modules. Moreover, it has the extra structure of a dg-algebra. However, instead of considering the ``Proj" category used in \cite{rosen}, we take the point of view introduced at the end of Section~\ref{s:sheaves}, and study a triangulated subcategory of $\OO_{q}$ dg-modules. To begin, we define the algebra $S_{q}$. 

Recall that in the classical setting the symmetric algebra can be realized as the direct sum of the irreducible representations of $\text{SU}(2)$. Taking this as motivation, we define the algebra $S_{q}$ as follows:
$$S_{q} = \bigoplus_{k \in \Z_{2h}} V_{k}$$
where we set $V_{k} = 0$ for $h-1\leq k \leq 2h-1$.

Note that $S_{q}$ comes with a (homogeneous) $\Z_{2h}$ grading given by $(S_{q})_{n} = V_{n}$. 

\begin{remark}
The reason for choosing a $\Z_{2h}$ grading is as follows. In the classical case, the functor given by tensoring with $\OO (2)$ gives a Coxeter element in the corresponding root system. In the quantum case the root system is finite, with Coxeter number $h$, so for tensoring by $\OO_{q} (2)$ to give a Coxeter element we need $\OO_{q} (2h) = \OO_{q}$.
\end{remark}

Define multiplication by the projection map $V_{n} \otimes V_{m} \to V_{n+m}$. Then $S_{q}$ is a graded algebra with unit given by $\I \hookrightarrow S_{q}$. It is generated in degree 1 by $V_{1}$.

For $n \in \Z_{2h}$ define $S_{q} (n)$ by setting $S_{q}(n)_{k} = V_{n+k}$.

Recall that in $\C$ we have $V_{h-2} \otimes V_{n} = V_{h-2-n}$ for $0\leq n \leq h-2$, and $V_{h-2} \otimes V_{n} = 0$ for $h-1 \leq n \leq 2h-1$.

Note that the algebra $S_{q}$ does not satisfy the analogue of the triangle ~\ref{e:sseq}, because in degree $h-2$ we get $0\to V_{h-2} \to 0 \to 0 \to 0$, which is not a short exact sequence and hence does not give a triangle in the corresponding derived category.

We now define a $\Z_{2}$-graded complex $\OO_{q}$, which should be thought of as an analogue of the structure sheaf $\OO$. In particular it gives an exact triangle analogous to the one coming from the short exact sequence $0\to \OO \to \OO(1) \otimes V \to \OO (2) \to 0$.

Consider $S_{q}$ and $S_{q}(h) \otimes V_{h-2}$ as complexes concentrated in degree zero. They are both naturally $\Z_{2h}$-graded $S_{q}$-modules. Set $\OO_{q} = S_{q} \oplus (S_{q} (h) \otimes V_{h-2}) [1]$, considered as a complex with trivial differential and $\Z_{2}$ homological grading.

\begin{remark}
The reason for considering $\Z_{2}$-graded complexes also comes from the corresponding root system. We would like the shift functor $[1]$ to act by longest element in the corresponding Weyl group, so we need $[2]$ to be the identity.
\end{remark}

Let $\D ( \C)$ be the derived category of $\C = \Rep (U_{q} (\s_{2}) )$, and let $\D (\C) / T^{2}$ be the corresponding 2-periodic derived category. We shall view $\OO_{q}$ as an object in $\D (\C) / T^{2}$.

\begin{prop}\label{p:triangle}
There is a triangle in $\D (\C) / T^{2}$ of the form
\begin{equation}\label{e:O}
\OO_{q} \to \OO_{q} (1) \otimes V_{1} \to \OO_{q}(2) \to \OO_{q} [1].
\end{equation}
Applying the twist functor $(n)$ gives triangles of the form
\begin{equation}\label{e:O(n)}
\OO_{q}(n) \to \OO_{q} (n+1) \otimes V_{1} \to \OO_{q}(n+2) \to \OO_{q} [1].
\end{equation}
\end{prop}

\begin{proof}
We will look at Equation~\ref{e:O} degree by degree. (The reader may wish to refer to Example~\ref{e:A4}.)\\

For $0\leq k \leq h-4$ the complexes involved are all concentrated in degree $0$. Since $V_{k+1} \otimes V_{1} = V_{k+2} \oplus V_{k}$ for $k$ in this range, we get short exact sequences in $\C$ of the form
$$0 \to V_{k} \to V_{k+2} \otimes V_{1} \to V_{k+2} \to 0,$$
and hence a triangle in degree $k$.

For $k=h-3$ we have that $V_{h-2} \otimes V_{1} = V_{h-3}$, so again we have a short exact sequence
$$0 \to V_{h-3} \to V_{h-2} \otimes V_{1} \to 0 \to 0,$$
which gives a triangle in degree $k$.

For $k=h-2$ we have
$$ V_{h-2} \to 0 \to V_{h-2}[1] \to V_{h-2} [1]$$
which defines a triangle in degree $k$. (Note that this explains the necessity of complexes to define $\OO_{q}$. In $\C$ the third term is $0$, and hence would not define a short exact sequence.)

For $k=h-1$, the complex is concentrated in homological degree 1, the triangle comes from the short exact sequence $0\to 0 \to V_{h-2} \otimes V_{1} \to V_{h-3} \to 0$, since $V_{h-2} \otimes V_{1} = V_{h-3}$.

For $h\leq k \leq 2h-4$ the complexes are concentrated in degree 1, and the triangles come from short exact sequences of the form
$$0 \to V_{h-2-k+1} \to V_{h-2-k} \otimes V_{1} \to V_{h-2-k-1} \to 0.$$

For $k=2h-3$ the complex is concentrated in degree 1, and the triangle comes from the short exact sequence
$$0\to V_{1} \to \I \otimes V_{1} \to 0 \to 0.$$

For $k=2h-2$ the triangle is given by
$$\I [1] \to 0 \to \I \to \I.$$

For $k=2h-1$ the complex is concentrated in degree 0, and the triangle comes from the short exact sequence
$$0\to 0 \to \I \otimes V_{1} \to V_{1} \to 0.$$

For Equation~\ref{e:O(n)}, apply the twist by $n$ to the triangle in Equation~\ref{e:O}. Since twisting is an exact functor, it takes triangles to triangles.
\end{proof}

\begin{example}\label{e:A4}
Let $h=5$, so we have non-zero simple objects $V_{0}, V_{1}, V_{2}, V_{3}$.

The triangle $\OO_{q} \to \OO_{q} (1) \otimes V_{1} \to \OO_{q}(2) \to \OO_{q} [1]$ is given below. The rows correspond to homogeneous degree, while the columns correspond to homological degree.

\begin{equation*}
\begin{bmatrix}
\I & 0 \\
V_{1} & 0 \\
V_{2} & 0 \\
V_{3} & 0 \\
0 & 0 \\
0 & V_{3} \\
0 & V_{2} \\
0 & V_{1} \\
0 & \I \\
0 & 0 
\end{bmatrix} 
\to
\begin{bmatrix}
V_{1} \otimes V_{1} & 0 \\
V_{2} \otimes V_{1} & 0 \\
V_{3} \otimes V_{1} & 0 \\
0 & 0 \\
0 & V_{3} \otimes V_{1} \\
0 & V_{2} \otimes V_{1} \\
0 & V_{1}\otimes V_{1} \\
0 & \I \otimes V_{1} \\
0 & 0 \\
\I \otimes V_{1} & 0 
\end{bmatrix} 
\to
\begin{bmatrix}
V_{2} & 0 \\
V_{3} & 0 \\
0 & 0 \\
0 & V_{3} \\
0 & V_{2} \\
0 & V_{1} \\
0 & \I \\
0 & 0 \\
\I & 0 \\
V_{1} & 0 
\end{bmatrix} 
\to
\begin{bmatrix}
0 & \I \\
0 & V_{1} \\
0 & V_{2} \\
0 & V_{3} \\
0 & 0 \\
V_{3} & 0 \\
V_{2} & 0 \\
V_{1} & 0 \\
\I  & 0\\
0 & 0 
\end{bmatrix} 
\end{equation*}
\end{example}

\begin{prop}
The complex $\OO_{	q}$ has the natural structure of a dg-algebra in $\C$ (with $\Z_{2}$ homological grading, and $\Z_{2h}$ homogeneous grading).
\end{prop}
\begin{proof}
Define a multiplication on $\OO_{q}$ by $V_{n}^{k} \otimes V_{m}^{l} \to V_{n+m}^{k+l}$, where by $V_{n}^{k}$ we mean the object $V_{n}$ in homological degree $k$. Note that this respects the $\Z_{2h}$ grading. To see this note that $V_{n}^{1}$ sits  in $\OO_{q}$ in homogeneous degree $h+n$, so that $(h+n) + (h+m) = 2h + (n+m) = n+m \ (\text{mod } 2h)$ and $V_{n}^{1} \otimes V_{m}^{1} \to V_{n+m}^{0}$.

The differential is trivial, so it follows easily that $\OO_{q}$ is a dg-algebra.
\end{proof}

Let $\OO_{q}-\Mod$ be the category of $\Z_{2h}$-graded dg $\OO_{q}$-modules in $\C$, and $\D (\OO_{q})$ be its derived category.

\begin{defi}
Let $\D (\PP^{1}_{q}) \subset \D (\OO_{q})$ be the full subcategory of objects $M$, such that $M \to M(1) \otimes V_{1} \to M(2) \to M[1]$ is a triangle, where the maps come from Equation~\ref{e:O}. We call $\D ( \PP^{1}_{q})$ the derived category of coherent sheaves on $\PP^{1}_{q}$.
\end{defi}

\begin{remark}
Note that the homogeneous $\Z_{2h}$ grading forces us to consider an alternative definition for sheaves, since the usual equivalence on $S$-modules of ``for sufficiently high degree" makes no sense in this setting. 
\end{remark}

\begin{prop}
For all $n\in \Z_{2h}$, $\OO_{q} (n) \in \D (\PP^{1}_{q})$.
\end{prop}

\begin{proof}
This is just Proposition~\ref{p:triangle} rephrased.
\end{proof}
\section{Equivariant Sheaves on $\PP^{1}_{q}$}\label{s:qsheaves}

In this section we introduce the notion of $G$-equivariant sheaf on $\PP^{1}_{q}$, where $G$ is a finite subgroup of $U_{q} (\s_{2})$.

Recall that a $G$-equivariant sheaf on classical $\PP^{1}$ can be though of as an equivariant $S$-module. A G-equivariant $S$-module is a graded $S$-module $M = \oplus M_{k}$, with a $G$ action preserving the grading, such that the action map $S \otimes M \to M$ is a map of $G$ modules. In particular, each $M_{k}$ is a $G$-module.

Now consider the quantum case. Let $G$ be a finite subgroup of $U_{q} (\s_{2})$. Let $\MM$ denote the corresponding module category over $\C$.

Consider the category of $\Z_{2h}$-graded, $G$-equivariant, $\OO_{q}$ dg-modules. Denote this category by $\OO_{q} - \MM$. Let $\D_{G} (\OO_{q})/T^{2}$ be the corresponding 2-periodic derived category.

An object $M \in \D_{G} (\OO_{q})/T^{2}$ has two gradings. The ``homological" $\Z_{2}$ grading (appearing as superscripts), and the ``homogeneous" $\Z_{2h}$ grading (appearing as subscripts). An object in this category can be thought of as a $\Z_{2}$-graded complex of $G$-modules with a $\Z_{2h}$ grading $M^{\bullet} = \oplus M^{\bullet}_{k}$, together with a dg-map $\OO_{q} \otimes M^{\bullet} \to M^{\bullet}$ that respects the $\Z_{2h}$-gradings of $\OO_{q}$ and $M^{\bullet}$.

We denote by $M^{k}_{j}$ the $j$-th homogeneous component of $M^{\bullet}$ in homological degree $k$.

To make the $\OO_{q}$ action on $M^{\bullet}$ precise, we should really consider the corresponding object $F(\OO_{q})$ which is an honest complex in $\MM$. Recall that $F$ is a tensor functor, so that it makes sense to use $F$ to give an action of $\OO_{q}$. However, we will typically drop the $F$, and just write $\OO_{q}$ for simplicity. Note that $F(\OO_{q})$ has the natural structure of $G$-equivariant $\OO_{q}$-module, with the grading given in the previous section.

Define the ``twist" functor on $\D_{G}(\OO_{q})/T^{2}$ in the usual way: $M(n)_{k} = M_{n+k}$. This functor is exact.

Define the ``global sections functor" $\Gamma : \D_{G} (\OO_{q})/T^{2} \to \MM$ by $\Gamma (M) = H^{0} (M_{0})$.

Define a functor $\widetilde{  } : \MM \to \D_{G} (\OO_{q})/T^{2}$ by setting $\widetilde{X} = \OO_{q} \otimes X$ (with grading given by $\widetilde{X}^{k}_{n} = (\OO_{q})^{k}_{n} \otimes X$).

\begin{lemma}
$\Gamma (\widetilde{X}) = X$.
\end{lemma}

\begin{proof}
By definition $\Gamma (\widetilde{X}_{0}) = H^{0}((\OO_{q} \otimes X)_{0}) = \I \otimes X = X$.
\end{proof}
In analogy with the description of $\D_{G} \subset \D(S-\Mod_{G})$ given in Definition~\ref{d:DG}, we make the following definition.

\begin{defi}
Let $\D_{G} (\PP^{1}_{q}) \subset \D_{G}(\OO_{q})/T^{2}$ be the full triangulated subcategory such that 
\begin{equation}\label{e:fundrel}
 M \to V_{1} \otimes M(1) \to M(2) \to M[1]
\end{equation}
is an exact triangle. (The first map comes from the map $V_{1} \otimes M \to M (1)$ and rigidity, and the second map $V_{1} \otimes M(1) \to M(2)$ is given by the action of $V_{1} \in \OO_{q}$. Here we fix a morphism $V_{1} \otimes V_{1} \to V_{0}$.)
\end{defi}
\section{Objects in $\D_{G}(\PP^{1}_{q})$}\label{s:objects}

Recall that $\MM$ has finitely many simple objects $X_{i}$, indexed by vertices $i \in \Gamma$, where $\Gamma$ is the corresponding $A,D,E$ Dynkin diagram.

For each pair $(i,n) \in \Gamma \times \Z_{2h}$ we can define an object $X_{i} (n) \in \D_{G} ( \PP^{1}_{q})$ by setting $X_{i} (n) = \OO_{q} (n) \otimes X_{i}$. The $\Z_{2h}$ grading is given by $(X_{i} (n))_{k} = \OO_{q} (n)_{k} \otimes X_{i}$. These should be thought of as analogues of the indecomposable locally free sheaves $\OO (n) \otimes X_{i}$ in the classical case.

\begin{prop}
There is a triangle $\OO_{q} \otimes X_{i} \to \OO_{q} (1) \otimes V_{1} \otimes X_{i} \to \OO_{q} (2) \otimes X_{i}$, so that $X_{i} (n)$ is in fact an object in $\D_{G} (\PP^{1}_{q})$.
\end{prop}

\begin{proof}
Tensoring with the triangle from Equation~\ref{e:O} with the object $X_{i}$ gives us the desired triangle. Twisting by $n$ shows that $X_{i} (n) \in \D_{G} (\PP^{1}_{q})$.
\end{proof}

Again, note that what is written is a bit ambiguous. To be precise, we should take the triangle from Equation~\ref{e:O}, apply the exact functor $F$ to get an $\OO_{q}$-module in $\MM$, then tensor with $X_{i}$ and twist by $n$ to get $X_{i}(n)$. However, this makes the notation rather cumbersome.

Note that by Theorem~\ref{t:qmckay} we can write $V_{1} \otimes X_{i} = \bigoplus_{j-i} X_{j}$. Hence we can rewrite the triangle $\OO_{q} \otimes X_{i} \to \OO_{q} (1) \otimes V_{1} \otimes X_{i} \to \OO_{q} (2) \otimes X_{i} \to \OO_{q} \otimes X_{i} [1]$ as 
\begin{equation}\label{e:AR}
\OO_{q} \otimes X_{i} \to \bigoplus_{i-j \text{ in } \Gamma} \OO_{q} (1) \otimes X_{j} \to \OO_{q} (2) \otimes X_{i} \to \OO_{q} \otimes X_{i} [1].
\end{equation}
Hence for every $(i,n) \in \Gamma \times \Z_{2h}$ we get a triangle
\begin{equation}\label{e:qfundrel}
X_{i} (n) \to  \bigoplus_{i-j \text{ in } \Gamma} X_{j} (n+1) \to X_{i}(n+2) \to X_{i} (n) [1].
\end{equation}

\begin{prop}\label{p:hom}
Let $X_{i} (n)$ be as above, let $\hh$ be a height function and let $\Gamma_{\hh}$ be the corresponding slice.
\begin{enumerate}
\item $\Hom ( X_{i} (n) , X_{i} (n) ) = \mathbb{K}$
\item $\Ext^{1} ( X_{i} (n) , X_{i} (n) ) = 0$.
\item $\Hom ( X_{i} (n) , X_{j} (n+1) ) = 
\begin{cases}
\mathbb{K} \text{ if } i-j \text{ in } \Gamma \\
0 \text{ otherwise.}
\end{cases}
$
\item If $(i,n), (j,m) \in \Gamma_{\hh}$, then $\Ext^{k} (X_{i} (n) , X_{j} (m)) = 0$ for $k\geq 1$.
\item If there is a path $(i,n) \to (j,m)$ in $\Gamma_{\hh}$, then $\Hom (X_{j} (m), X_{i} (n)) = 0$.
\end{enumerate}
\end{prop}

To prove this, we first need a Lemma.

\begin{lemma}\label{l:hom}
Let $f: X_{i} (n) \to X_{j} (m)$ be a morphism of equivariant $\OO_{q}$-modules, and let $f_{k}$ denote the degree $k$ component of $f$. Then $f$ is completely determined by $f_{2h-n} : \I \otimes X_{i} \to V_{m-n} \otimes X_{j}$.
\end{lemma}

\begin{proof}
To make indexing simpler, consider only the case $\OO_{q} \otimes X_{i} \to \OO_{q} (m) \otimes X_{j}$, with $n=0$ and $m \in \Z_{2h}$. To obtain the more general statement, twist both sides by $n$. 

Since $f$ is a map of equivariant $\OO_{q}$ modules the following diagram is commutative for all $k$.

$$
\xymatrix{
\I \otimes X_{i} \ar[rr]^{f_{0}} \ar[d] & & V_{m} \otimes X_{j} \ar[d] \\
V_{k} \otimes X_{i} \ar[rr]^{id\otimes f_{0}} \ar[d]^{id} & & V_{k} \otimes V_{m} \otimes X_{j} \ar[d]_{m} \\
V_{k} \otimes X_{i} \ar[rr]^{f_{k}} & & V_{m+k} \otimes X_{j}
}$$

This deals with the maps in homological degree 0.

Similarly, for the maps in homological degree 1, we have that the following diagram is commutative for all $k$.

$$
\xymatrix{
\I \otimes X_{i} \ar[rr]^{f_{0}} \ar[d] & & V_{m} \otimes X_{j} \ar[d] \\
(V_{h-2} \otimes V_{h-2-k}) \otimes X_{i} \ar[rr]^{id\otimes f_{0}} \ar[d]^{id} & & (V_{h-2} \otimes V_{h-2-k}) \otimes V_{m} \otimes X_{j} \ar[d]_{m} \\
(V_{h-2} \otimes V_{h-2-k}) \otimes X_{i} \ar[rr]^{f_{k}} & & (V_{h-2} \otimes V_{h-2-(k+m)}) \otimes X_{j}
}$$

Noting that in both cases the bottom left arrow is the identity, we see that $f_{k} = m (\id \otimes f_{0})$.
\end{proof}

We now complete the proof of Proposition~\ref{p:hom}.
\begin{proof} \
\begin{enumerate}
\item Let $f_{k} : X_{i} (n)_{k} \to X_{i} (n)_{k}$ denote the morphism in degree $k$. Note that in degree 0 we just have $X_{i} \to X_{i}$, which is 1 dimensional. By Lemma~\ref{l:hom} the choice of the map $f_{0}$ completely determines $f_{k}$, so $\Hom (X_{i} (n) , X_{i} (n) )$ is 1 dimensional.
\item Note that there are no morphisms $X_{i} (n) \to X_{i} (n) [1]$, so $$\Ext^{1} (X_{i} (n), X_{i} (n)) = \Hom (X_{i} (n) , X_{i} (n) [1]) = 0.$$
\item As above, the map in degree $2h-n$ completely determines $f$. In degree $2h-n$ we have $\I \otimes X_{i} \to V_{1} \otimes X_{j} = \bigoplus_{l-j} X_{l}$, which gives the result.
\item First note that $(\OO_{q} \otimes X_{i})^{0}$ is concentrated in homogeneous degrees $0 \leq k \leq h-2$ and $(\OO_{q}(1) \otimes X_{j})^{1}$ is concentrated in homogeneous degrees $h-1 \leq k \leq 2h-3$. Similarly, $(\OO_{q} \otimes X_{i})^{1}$ is concentrated in homogeneous degrees $h \leq k \leq 2h-1$ and $(\OO_{q}(1) \otimes X_{j})^{0}$ is concentrated in homogeneous degrees $ -1 \leq k \leq h-3$. Hence $\Hom (\OO_{q} \otimes X_{i} , (\OO_{q}(1) \otimes X_{j}) [1] ) = 0$. Twisting by $n$ gives $\Hom (X_{i} (n) , X_{j} (n+1) [1]) = 0$, hence $\Ext^{1} (X_{i} (n) , X_{j} (n+1) ) = 0$.\\
A similar argument comparing the homogeneous components of $X_{i} (n)$ and $X_{j} (n)$ show that $\Ext^{1} (X_{i} (n) , X_{j} (n) ) = 0$.\\
Together, this shows that if $(i,n), (j,m) \in \Gamma_{\hh}$ and $\Om_{\hh}$ is bipartite, then $\Ext^{1} (X_{i} (n) , X_{j} (m))= 0$. (If $\Om_{\hh}$ is bipartite, then $m=n$ or $m=n+1$.)\\
To prove the general case, use Theorem~\ref{t:tilting} in the case that $\hh$ is bipartite. The resulting equivalence of categories will give the result by well-know statements about $\Ext$ for quiver representations. (Note the Proof of Theorem~\ref{t:tilting} in the bipartite case does not depend on the full statement of this Proposition, so we are not making a circular argument here.)

\item Since the maximum length of a path in $\Gamma_{\hh}$ is $h-2$, we get that $m-n \leq h-2$, and $k=2h-(m-n) > h-1$. Hence $(X_{i} (n))^{0}_{2h-m} = (\OO_{q})^{0}_{k} \otimes X_{i}=0$, and so $f_{2h-m} = 0$. By Lemma~\ref{l:hom} $f : X_{j} (m) \to X_{i} (n)$ must also be zero.

\end{enumerate}
\end{proof}

\section{Representations of $\Gammahat$}\label{s:catD}

Recall the quivers  $\Gammahat$, $\Gamma \times \Z_{2h}$, and $\Gammahat_{cyc}$. In \cite{kirillovthind2} we studied a triangulated category of complexes of representations of $\Gammahat$, denoted by $\D$. For a vertex $q$ and object $X$, we denote by $X(q)$ the complex of vector spaces attached to $q$ by the object $X$. The category $\D$ had objects which were complexes of representations of $\Gammahat$ satisfying the ``fundamental relation" that for every vertex $(i,n) \in \Gammahat$ there is an isomorphism 
$$X(i,n+2) \simeq \text{Cone} \{ X(i,n) \to \bigoplus_{j-i} X(j,n+1) \}.$$ Alternatively, we can study representations of $\Gamma \times \Z$ satisfying the fundamental relation. Denote this category by $\D(\Gamma \times \Z)$. Since $\Gamma \times \Z = \Gammahat^{0} \sqcup \Gammahat^{1}$, and $\Gammahat^{0} \simeq \Gammahat^{1}$ as quivers, an object consists of a pair $(M_{1} , M_{2}) \in \D \times \D$. So $\D(\Gamma \times \Z) = \D \times \D$.

We then considered the 2-periodic quotient $\C (\Gammahat)= \D /T^{2}$. Objects here can be thought of as $\Z_{2}$-graded complexes of representations of $\Gammahat_{cyc}$ satisfying the same fundamental relation. Similarly, we can consider representations of $\Gamma \times \Z_{2h}$, satisfying the fundamental relation. Denote this category by $\C(\Gamma \times \Z_{2h})$, and note that $\C (\Gamma \times \Z_{2h}) = \C(\Gammahat) \times \C(\Gammahat)$.

There is a functor $\tau : \C(\Gammahat) \to \C(\Gammahat)$ defined by $\tau (X) (q) = X (\tau^{-1}_{\Gammahat} q)$, where $\tau_{\Gammahat} : \Gammahat_{cyc} \to \Gammahat_{cyc}$ is the translation defined in Section~\ref{s:mckay}.

We summarize the important properties of the category $\C(\Gammahat)$ in the following proposition. For more details (and proof of the following Proposition) see \cite{kirillovthind2}.

\begin{prop}\label{p:catC}
Let $\K (\C(\Gammahat))$ denote the Grothendieck group of $\C(\Gammahat)$. \\
Set $< \F , \G > = \dim \Hom (\F, \G) - \dim \Ext^{1} (\F , \G)$ and $( \F, \G ) = < \F , \G > + < \G, \F >$.
\begin{enumerate}
\item $\K (\C(\Gammahat))$ can be identified with the root lattice of $\Gamma$, with bilinear form given by the form $( \cdot \ , \cdot )$ above. Moreover, the set $\Ind = \{ [X] \in \K(\C(\Gammahat)) \ | \ X \text{ is indecomposable in  } \C(\Gammahat) \}$ can be identified with the set of roots.
\item The map induced by the equivalence $X \mapsto \tau X$ on $\K ( \C(\Gammahat) )$, is a Coxeter element for the corresponding root system.
\item For each vertex $v \in \Gammahat_{cyc}$, there is a unique, natural, indecomposable object $X_{v}$, and these give a full list of indecomposable objects (up to isomorphism).
\item $\Hom (X_{v} , X_{v^{\prime}}) = \Path_{\Gammahat} (v, v^{\prime} ) / J$ where $J$ is the mesh ideal from Section~\ref{s:mckay}.
\item The map $\Ind \to \Gammahat_{cyc}$ given by $[X_{v}] \mapsto v$, identifies $\Gammahat_{cyc}$ as the Auslander-Reiten quiver of $\C(\Gammahat)$, and identifies the Coxeter element with the translation $\tau_{\Gammahat}$.
\item For each $v$ there is a triangle $$X_{v} \to \bigoplus_{v\to v^{\prime}} X_{v^{\prime}} \to X_{\tau_{\Gammahat} v} \to X_{v} [1].$$ These triangles give a full list of relations in $\K (\C(\Gammahat))$.
\item $\C(\Gammahat)$ has Serre Duality: $\Hom (X,Y) \simeq \big{(}\Ext^{1} (Y, \tau X) \big{)}^{*}$.
\end{enumerate}
\end{prop}

We shall define an exact functor between $\D_{G} (\OO_{q})$ and $\C(\Gamma \times \Z_{2h})$. 

Consider an object $M \in \D_{G} (\OO_{q})$. Write $M=\bigoplus_{k \in \Z_{2h}} M^{\bullet}_{k}$. Then $M^{\bullet}_{k} = \bigoplus_{i} M^{\bullet}_{i,k} \otimes X_{i}$ where $M^{\bullet}_{i,k}$ is the multiplicity space for $X_{i}$ in $M^{\bullet}_{k}$. So we have the following decomposition of $M$:
\begin{equation}\label{e:objectM}
M = \bigoplus_{(i,k) \in \Gamma \times \Z_{2h}} M^{\bullet}_{i,k} \otimes X_{i}.
\end{equation}
Consider the map $M^{\bullet} \to M^{\bullet}(1) \otimes V_{1}$. In degree $k$ this map looks like $M^{\bullet}_{k} \to M^{\bullet}_{k+1} \otimes V_{1}$. The right hand side has the following $X_{i}$-th component:
\begin{align*} 
[ M^{\bullet}_{k+1} \otimes V_{1} ]_{i} &= [ \bigoplus_{j} (M^{\bullet}_{j,k+1} \otimes X_{j}) \otimes V_{1} ]_{i}\\
&= [ \bigoplus_{j} M^{\bullet}_{j,k+1} \otimes (X_{j} \otimes V_{1} ) ]_{i} \\
&= [ \bigoplus_{j} M^{\bullet}_{j,k+1} \otimes (\bigoplus_{l-j} X_{l}) ]_{i} \\
&= [ \bigoplus_{j} \bigoplus_{l-j} M^{\bullet}_{j,k+1} \otimes X_{l} ]_{i} \\
&= \bigoplus_{j-i} M^{\bullet}_{j,k+1} \otimes X_{i}. \\
\end{align*}
From this we see that the map $M^{\bullet} \to M^{\bullet}(1) \otimes V_{1}$ gives linear maps $m_{i,k} : M^{\bullet}_{i,k} \to \bigoplus_{j-i} M^{\bullet}_{j,k+1}$.

Define an exact functor $\Phi$ from $\D_{G} (\OO_{q}) \to \D (\Gamma \times \Z_{2h})$, where $\D (\Gamma \times \Z_{2h})$ is the derived category of representations of $\Gamma \times \Z_{2h}$, as follows:

\begin{align*}
\Phi (M^{\bullet}) &= \bigoplus_{(i,k) \in \Gamma \times \Z_{2h}} \RHom (X_{i} , M^{\bullet}_{k}) \\
&= \bigoplus_{ (i,k) \in \Gamma \times \Z_{2h}} M^{\bullet}_{i,k}
\end{align*}

with the maps along edges given by $m_{i,k} :M^{\bullet}_{i,k} \to \bigoplus_{j-i} M^{\bullet}_{j,k+1}$, induced from $M^{\bullet} \to M^{\bullet}(1) \otimes V_{1}$.

Restrict this functor to $\D_{G} (\PP^{1}_{q})$. Then since an object $M \in \D_{G} ( \PP^{1}_{q} )$ satisfies $M \to M(1) \otimes V_{1} \to M(2) \to M[1]$ is a triangle,  and $\RHom$ is exact, we see that $\Phi (M)$ satisfies the fundamental relation. In particular, there is a triangle $M_{(i,n)} \to \bigoplus_{j - i} M_{(j, n+1)} \to M_{(i, n+2)} \to M_{(i,n)} [1]$, and isomorphim $M_{(i,n+2)} \simeq Cone (m_{i,k})$. Hence $\D_{G} ( \PP^{1}_{q})$ maps into $\C(\Gamma \times \Z_{2h})$. I.e. we get a functor $\Phi : \D_{G} (\PP^{1}_{q}) \to \C(\Gamma \times \Z_{2h})$.

Now define a functor $\Psi : \C(\Gamma \times \Z_{2h}) \to \D_{G} ( \PP^{1}_{q})$ by

$$\Psi (Y) = \bigoplus_{(i,k) \in \Gamma \times \Z_{2h}} Y_{i,k} \otimes X_{i} (k)$$ 

where we view $Y_{i,k}$ as a multiplicity space. The maps $Y_{i,k} \to \bigoplus Y_{j,k+1}$ induce maps $\Psi (Y) \to \Psi (Y) (1) \otimes V_{1}$. Rigidity then gives the action map $\Psi (Y) \otimes V_{1} \to \Psi (Y) (1)$. Since $Y$ satisfies the fundamental relation, $\Psi (Y)$ satisfies $\Psi(Y) \to \Psi (Y) (1) \otimes V_{1} \to \Psi(2) \to \Psi(Y) [1]$ is a triangle.

\begin{thm}\label{t:equiv}
The functor $\Phi$ defines an equivalence between $\D_{G} (\PP^{1}_{q})$ and $\C(\Gamma \times \Z_{2h})$, with inverse $\Psi$.
\end{thm}

\begin{proof}
Follows from the definitions of $\Phi$ and $\Psi$.
\end{proof}

\begin{cor}
The objects $X_{i} (n)$ form a complete list of indecomposable objects in $\D_{G} (\PP^{1}_{q})$.
\end{cor}

\begin{proof}
The category $\C(\Gammahat)$ has indecomposables $X_{(i,n)}$ for $(i,n) \in \Gammahat$, and the functor $\Phi$ takes $X_{i} (n)$ to $X_{(i,n)}$. This can be seen by using Lemma 7.1 in \cite{kirillovthind2} and looking at homogeneous degrees $0,1$.
\end{proof}

Since $\C(\Gamma \times \Z_{2h}) = \C(\Gammahat) \times \C(\Gammahat)$, we define $\D_{\overline{G}} (\PP^{1}_{q}) = \Psi^{-1} ( \C(\Gammahat), 0)$, the full triangulated category generated by inverse images of objects of the form $(Y, 0) \in \C(\Gamma \times \Z_{2h})$.

\begin{cor}\label{c:indec} \
\begin{enumerate}
\item The objects $X_{i} (n)$ with $(i,n) \in \Gammahat_{cyc}$ form a complete list of indecomposable objects in $\D_{\overline{G}} (\PP^{1}_{q})$.
\item The equivalence $\Phi$ identifies the functor $X \mapsto X(2)$ with $\tau$.
\item The equivalence $\Phi$ takes the triangle $$X_{i} (n) \to  \bigoplus_{i-j \text{ in } \Gamma} X_{j} (n+1) \to X_{i}(n+2) \to X_{i} (n) [1]$$ 
from Equation~\ref{e:qfundrel} to the triangle $$X_{(i,n)} \to \bigoplus_{i-j \text{ in } \Gamma} X_{(j,n+1)} \to X_{(i,n+2)} \to X_{(i,n)} [1].$$ 
\item The map $\Ind \to \Gammahat_{cyc}$ given by $X_{i} (n) \mapsto (i,n)$, identifies $\Gammahat_{cyc}$ as the Auslander-Reiten quiver of $\D_{\overline{G}} (\PP^{1}_{q})$.
\end{enumerate}
\end{cor}

\section{Equivalences of Categories}\label{s:tilting}

In this section we construct equivalences $\D_{\overline{G}} (\PP^{1}_{q}) \to \D^{b} (\Gamma , \Om_{h}^{op})/T^{2}$, for every height function $\hh$. This allows us to prove the analogue of Theorem~\ref{t:Gsheaves} for $\PP^{1}_{q}$.

Let $\hh$ be a height function. Define a ``restriction" functor $\rho_{\hh} : \D_{\overline{G}} (\PP^{1}_{q}) \to \D^{b} (\Gamma , \Om_{h})/T^{2}$, by 
$$\rho_{\hh} (M) = \bigoplus_{i \in \Gamma} M_{(i, \hh (i) )},$$
where the maps along edges come from $m_{(i, \hh (i))} : M_{(i, \hh (i))} \to \bigoplus_{j -i} M_{(j, \hh (j))}$. Note that by definition, $\rho$ is exact.

\begin{thm}\label{t:tilting}
The functor $\rho_{\hh}$ is an equivalence of categories.
\end{thm}

To prove this Theorem, we first need the following Lemma.

\begin{lemma}\label{l:determine}
\ 
\begin{enumerate}
\item Any object $M \in \D_{\overline{G}} (\PP^{1}_{q})$ is determined (up to isomorphism) in homogeneous degrees $0,1$.
\item Any object $M \in \D_{\overline{G}} (\PP^{1}_{q})$ is determined (up to isomorphism) by the complexes $M^{\bullet}_{(i, \hh (i))}$ for any height function $\hh$.
\item Any morphism $f : M \to N$ is determined by its values on $M^{\bullet}_{(i, \hh (i))}$ for any height function $\hh$.
\end{enumerate}
\end{lemma}

\begin{proof} \
\begin{enumerate}
\item Since any $M \in \D_{\overline{G}} ( \PP^{1}_{q} )$ satisfies the condition that $M \to M(1) \otimes V_{1} \to M(2) \to M[1]$ is a triangle, we see that if we know $M_{0} , M_{1}$, and the map $M_{0} \to M_{1} \otimes V_{1}$, we have that $M_{2} \simeq \text{Cone} (M_{0} \to M_{1} \otimes V_{1})$. Inductively, we can recover $M_{k}$ for all $k$.
\item The proof is the same as for (1). (Note that (1) is just the special case that $\hh$ is the bipartite height function $\hh (i) = p(i)$.)
\item Recall that any height function $\hh$ gives an orientation $\Om_{\hh}$. Suppose that $f : M_{(i,\hh (i) )} \to N_{(i,\hh (i) )}$ is given for all $i$. Take $i \in \Gamma$ to be a source for $\Om_{\hh}$. Then using the isomorphisms $M_{i, \hh (i) + 2} \stackrel{\sim}{\to} Cone (m_{(i, \hh(i))}) $ and $N_{i, \hh (i) + 2} \stackrel{\sim}{\to} Cone (n_{(i, \hh(i))})$, there is a unique map making the diagram below commutative, which extends $f$ to $f_{(i, \hh (i) + 2 )} : M_{(i, \hh (i) +2)} \to N_{(i, \hh (i) +2)}  $.
$$
\xymatrix{
M_{(i, \hh(i) + 2)} \ar[d] \ar[r]^{f_{(i, \hh (i) + 2 )} } & M_{(i, \hh(i) + 2)}  \ar[d] \\
Cone (m_{(i, \hh(i))}) \ \ \ \ar[r]^{Cone( f_{(i, \hh (i) ) } ) } & \ \ \ Cone(n_{(i, \hh(i))}) \\
}$$
Continuing in this way it is possible to extend $f$ to all $(i, n) \in \Gammahat_{cyc}$, and hence for all $M_{i,k}$. 
\end{enumerate}
\end{proof}

We now complete the proof of Theorem~\ref{t:tilting}.
\begin{proof}
By Lemma~\ref{l:determine}, any object $M$ and any morphism $f: M \to N$ in $\D_{\overline{G}} (\PP^{1}_{q})$ is determined (up to isomorphism) by specifying them for a height function. Hence, given $Y \in  \D^{b} (\Gamma , \Om_{h})/T^{2}$, we can define an object $M_{Y}$ (unique up to isomorphism), which maps to $Y$ under $\rho_{\hh}$. Moreover, $\Hom_{\D_{\overline{G}} (\PP^{1}_{q})} (M, N) \simeq \Hom_{\D^{b} (\Gamma , \Om_{\hh}) /T^{2}} (R \rho_{\hh} M , R \rho_{\hh} N)$. Hence $\rho$ is an equivalence.
\end{proof}

The following Corollary summarizes the properties of $\D_{\overline{G}} (\PP^{1}_{	q})$ which generalize the properties of $\D_{\overline{G}} (\PP^{1})$ from Theorem~\ref{t:Gsheaves}.

\begin{cor}\label{c:main} \
\begin{enumerate}
\item Let $\K$ denote the Grothendieck group of $\D_{\overline{G}} (\PP^{1}_{	q})$. \\
Set $< \F , \G > = \dim \Hom (\F, \G) - \dim \Ext^{1} (\F , \G)$ and $( \F, \G )_{\PP^{1}_{q}} = < \F , \G > + < \G, \F >$.\\ 
Then $\K$ can be identified with the root lattice of $\Gamma$, the form $( \cdot \ , \cdot )_{\PP^{1}_{q}}$ above can be identified with its bilinear form, and the set $\Ind_{\PP^{1}_{q}} = \{ [X] \in \K \ | \ X \text{ is indecomposable in  } \D_{\overline{G}} (\PP^{1}_{q}) \}$ can be identified with the set of roots.
\item The map on $\K$ induced by the ``twist" $\F \mapsto \F (2)$, gives a canonical Coxeter element for this root system.
\item The classes of indecomposable sheaves are of the form $X_{i} (n) := X_{i} \otimes \OO_{q}(n)$ for $p(i)+n \equiv 0 \mod 2$. They are in bijection with vertices of the quiver $\Gammahat_{cyc}$ so that the arrows correspond to indecomposable morphisms $X_{i} (n) \to X_{j} (m)$. (I.e. $\Gammahat_{cyc}$ is the Auslander-Reiten quiver of $\D_{\overline{G}} (\PP^{1}_{q})$.
\item For any height function $\hh: \Gamma \to \Z_{2h}$, there is an equivalence of triangulated categories 
$\rho_{\hh} : \D_{\overline{G}} (\PP^{1}_{q}) \to \D^{b} (\Gamma, \Om_{\hh}) / T^{2}$
given by ``restriction".
\item If $i \in \Gamma$ is a source (or sink) for $\Om_{\hh}$, then the following diagram commutes:
$$\xymatrix{
&  \D^{b} (\Gamma ,\Om_{\hh}) / T^{2}   \ar[dd]^{S^{\pm}_{i}} \\
\D_{\overline{G}} (\PP^{1}_{q}) \ar[dr]_{\rho_{s^{\pm}_{i} \hh}} \ar[ur]^{\rho_{\hh}} & \\
& \D^{b} (\Gamma ,\Om_{s^{\pm}_{i} \hh}) / T^{2} \\
}$$
where $S^{\pm}_{i}$ is the derived BGP reflection functor.
\item The exact triangle
\begin{equation*}
\OO_{q} \to \OO_{q} \otimes V_{1} \to \OO_{q} (2) \to \OO_{q} [1]
\end{equation*}
gives an exact triangle of indecomposable objects
\begin{equation*}
X_{i} (n) \to \bigoplus_{i-j} X_{j} (n+1) \to X_{i} (n+2) \to X_{i} (n) [1]
\end{equation*}
which induces the full set of relations in $\K$.
\item $\Hom (X_{i} (n) , X_{j} (m)) = \Path_{\Gammahat} ((i,n), (j,m)) / J$ where $J$ is an explicitly described quadratic ideal.
\item $\D_{\overline{G}} (\PP^{1}_{q})$ has Serre Duality: $\Hom (X, Y) \simeq \big{(} \Ext^{1} (Y, X(2)) \big{)}^{*}$
\end{enumerate}

\end{cor}

\begin{remark}
One may define an object 
$$\T_{\hh}= \bigoplus_{(i,n) \in \Gamma_{\hh}} X_{i} (n)$$
for every height function $\hh$. There is an equivalence $\Phi_{\hh} :  \D_{\overline{G}} (\PP^{1}_{q}) \to \D^{b} (R)/T^{2}$, where $R=\End (\T_{\hh})$ and $\Phi_{\hh} (M) = \RHom (T,M)$. Moreover, $\End (\T_{\hh})$ can naturally be identified with the path algebra of the quiver $(\Gamma, \Om_{\hh}^{op})$, so we can view this equivalence as having values in $\D^{b} (\Gamma, \Om_{\hh}^{op} ) /T^{2}$. This can be thought of as a ``tilting" object, which would give an equivalence similar to that from Theorem~\ref{t:Gsheaves}. However, for simplicity, we chose to use a more direct equivalence.  
\end{remark}

\begin{proof} \
\begin{enumerate}
\item The equivalence $\Phi$, of Theorem~\ref{t:equiv}, gives an isomorphism $\K \simeq \K ( \C(\Gammahat) )$ and a bijection $\Ind_{\PP^{1}_{q}} \to \Ind_{\C(\Gammahat)}$. $\Phi$ also identifies $\Hom$ and $\Ext$ spaces. Hence the form $( \cdot , \cdot )_{\PP^{1}_{q}}$ is identified with $( \cdot , \cdot )_{\C(\Gammahat)}$. The result then follows from Proposition~\ref{p:catC}.

\item By Corollary~\ref{c:indec}, $\Phi$ identifies $X \mapsto X(2)$ with $\tau$. Proposition~\ref{p:catC} identifies $\tau$ as a Coxeter element for the corresponding root system, hence the map on $\K$ coming from $X \mapsto X(2)$ is a Coxeter element for the corresponding root system.

\item This is just Parts (1) and (3) of Corollary~\ref{c:indec}.
\item This is Theorem~\ref{t:tilting}.
\item This follows the proof of Theorem 9.6 in \cite{kirillovthind2}.
\item By Corollary~\ref{c:indec} the equivalence $\Phi$ takes the triangle
 $$X_{i} (n) \to  \bigoplus_{i-j \text{ in } \Gamma} X_{j} (n+1) \to X_{i}(n+2) \to X_{i} (n) [1]$$ 
to the triangle $$X_{(i,n)} \to \bigoplus_{i-j \text{ in } \Gamma} X_{(j,n+1)} \to X_{(i,n+2)} \to X_{(i,n)} [1].$$ By Proposition~\ref{p:catC}, such triangles induce a full set of relations in the Grothendieck group.
\item Again, $\Phi$ is an equivalence, and identifies $X_{i} (n) \in \D_{\overline{G}} (\PP^{1}_{q})$ with $X_{(i,n)} \in \C(\Gammahat)$. By Proposition~\ref{p:catC} we have $$\Hom_{\PP^{1}_{q}} (X_{i} (n) , X_{j}(m)) \simeq \Hom_{\C(\Gammahat)} (X_{(i,n)} , X_{(j,m)} ) = \Path_{\Gammahat} ((i,n) , (j,m)) / J$$ where $J$ is the mesh ideal.
\item The category $\C(\Gammahat)$ has Serre Duality given by  $\Hom (X,Y) \simeq \big{(} \Ext^{1} (Y, \tau X) \big{)}^{*}$. The equivalence $\Phi$ identifies $X \mapsto X(2)$ with $\tau$, so $\D_{\overline{G}} (\PP^{1}_{q})$ has Serre Duality, and it is given by $\Hom (X, Y) \simeq \big{(} \Ext^{1} (Y, X(2) )\big{)}^{*}$. 
\end{enumerate}
\end{proof}

\bibliographystyle{amsalpha}

\end{document}